\documentclass[12pt]{amsart}
\usepackage{amsmath,amscd}  
\usepackage{color}

\numberwithin{equation}{section}

\newtheorem{theorem}[subsection]{Theorem}
\newtheorem{lemma}[subsection]{Lemma}
\newtheorem{proposition}[subsection]{Proposition}
\newtheorem{corollary}[subsection]{Corollary}
\newtheorem{remark}[subsection]{Remark}

\newcommand{\C}{{\mathbb C}}
\newcommand{\Z}{{\mathbb Z}}
\newcommand{\N}{\mathbb{N}}

\newcommand{\SL}{{\text{SL}}}

\newcommand{\LM}{\mathop{\rm LM}}

\newcommand{\Pol}{{\mathcal P}}
\newcommand{\Res}{{\mathcal R}}
\newcommand{\Sml}{S_{\rm L}}

\newcommand{\B}[1]{{\mathcal{B}_{#1}}}
\newcommand{\Path}[1]{{\text {Path}_{#1}}}
\newcommand{\ba}{\overline{\alpha}}
\renewcommand\succ{{\mathop{\rm succ}}}

\newcommand{\supp}{\mathop{\rm supp}}
\newcommand{\G}{{\mathcal G}}
\newcommand{\lot}{{\text{ l.o.t.}}}

\title[Weitzenb\"ock derivations of nilpotency 3]{Weitzenb\"ock derivations of nilpotency 3}

\author[Wehlau]{David L. Wehlau}
\address{Department of Mathematics and Computer Science \\ \hfil\break\indent
        Royal Military College \\ King\-ston, Ontario, Canada \\ K7K 5L0
             }
\email{wehlau@rmc.ca}

\date{\today}
\subjclass[2010]{13N15; 13A50; 13P10; 14E07}

\keywords{Locally nilpotent derivations; algebra of constants; invariants of unitriangular transformations, Robert's isomorphism,
                  polarization, restitution, SAGBI bases}

\dedicatory{}

\begin{document}

\begin{abstract}  
 We consider a Weitzenb\"ock derivation $\Delta$ acting on a polynomial ring $R=K[\xi_1,\xi_2,\dots,\xi_m]$ over a field
   $K$ of characteristic 0.
 The  $K$-algebra $R^\Delta = \{h \in R \mid \Delta(h) = 0\}$ is called the algebra of constants.
 Nowicki considered the case where the Jordan matrix for $\Delta$ acting on $R_1$, the degree 1 component of
 $R$, has only Jordan blocks of size 2.  He conjectured 
  that a certain set generates $R^{\Delta}$ in that case.
 Recently Khoury, 
 Drensky and Makar-Limanov 
 and Kuroda 
  have given proofs of Nowicki's conjecture.
 Here we consider the case where the Jordan matrix for $\Delta$ acting on $R_1$ has only Jordan blocks of size at most
 3.   We use combinatorial methods to give a minimal set of generators $\mathcal G$ for the algebra of constants
 $R^{\Delta}$. Moreover, we show how our proof  yields an algorithm to express any $h \in R^\Delta$ as a polynomial
 in the elements of $\mathcal G$.   In particular, our solution shows how the classical techniques of polarization and restitution may be used to augment the techniques of SAGBI bases to construct generating sets for subalgebras.
 \end{abstract}

\maketitle
\tableofcontents

\section{Introduction}\label{intro}

    Let $K$ be a field of characteristic zero and let $R=K[\xi_1,\xi_2,\dots,\xi_m]$ be a polynomial ring over $K$ in $m$ variables each of degree 1.
    The ring $R$ has a standard $\N$-grading $R = \oplus_{d=0}^\infty R_d$ where $R_1 = \oplus_{i=1}^m K\xi_i$.
    A derivation $\Delta:R \to R$ is called {\it locally nilpotent} if, for every $a \in R$ there exists a positive integer $k$ such that $\Delta^k(a)=0$.
    Note that any locally nilpotent derivation of $R$ is a $K$-derivation.
    A derivation whose matrix representation on $R_1$ is a Jordan matrix with the zeros on the main diagonal is called a Weitzenb\"ock derivation.
    If $\Delta$ is a locally nilpotent derivation which restricts to $R_1$ then by an appropriate choice of basis we may suppose that
    $\Delta$ is a Weitzenb\"ock derivation.
    The kernel of $\Delta$ is a subalgebra of $R$ called the {\it algebra of constants} and denoted by $R^\Delta$.
    Weitzenb\"ock's Theorem~\cite{We}
    asserts that if $\Delta$ is a Weitzenb\"ock derivation then $R^\Delta$ is a finitely generated $K$-algebra.

    Recently the case where the Jordan matrix of $\Delta$ on $R_1$ consists consist of entirely $2\times 2$ blocks has been studied.
    Nowicki~\cite{N} conjectured that for this case $R^\Delta$ is generated by certain linear and quadratic polynomials.  This was proved by
    Khoury~\cite{Kh}, by Drensky and Makar-Limanov~\cite{DM} and also Kuroda~\cite{K}.
        Here we consider the case where the Jordan matrix of $\Delta$ on
    $R_1$ has blocks of size at most 3 and exhibit a set of generators for $R^\Delta$ in that case.  Furthermore, we give an algorithm
    for expressing any element of $R^\Delta$ as a polynomial in those generators.  
    
    Finding a finite SAGBI basis for the algebra of constants would provide an algorithm for expressing elements of  $R^\Delta$ as a polynomial in
    such a SAGBI basis.  We refer the reader to \cite[\S5.1]{CW} for a detailed discussion of SAGBI bases.
    We do not however provide a finite SAGBI basis here and indeed we suspect that none exists.  Nevertheless we are able to 
    combine the classical techniques of polarization and restitution with SAGBI basis techniques to provide an algorithm.  These ideas would seem to 
    apply to a wide range of subalgebras.   
    
    Rather than studying the kernel of $\Delta$ we may shift perspective and consider
    $\sigma := e^\Delta = 1 + \Delta + \Delta^2/2! + \Delta^3/3! + \dots$.
    Then $\sigma$ acts invertibly on $R_1$ and $R$.  We consider the infinite cyclic group $G$ of algebra automorphisms of $R$
    generated by $\sigma$.  Then $R^\Delta = R^G$, the ring of $G$ invariants.  We may use
    Robert's isomorphism to show that $R^\Delta = R^G \cong S^{\SL_2(K)}$ for a certain polynomial ring $S$ on which $\SL_2(K)$ acts linearly.
     This allows us to use the classical invariant theory of  $\SL_2(\C)$ to derive properties of $R^\Delta$.   For a discussion of this approach
     see \cite{shank conj} or \cite{B}.  For a modern treatment of  Robert's isomorphism see  \cite[Ch.~15 \S1.3 Theorem~1]{P}  and \cite{BK}.

    Another way to proceed is to use $\Z[\xi_1,\xi_2,\dots,\xi_m]$ in the role of $R$ and then reduce modulo a prime $p$.
    In this setting the cyclic group generated by $\sigma$ is $C_p$, the cyclic group of order $p$ and we study
    its ring of invariants $\Z/p\Z[\xi_1,\xi_2,\dots,\xi_r]^{C_p}$.  From this perspective we can compute
    $$\Z[\xi_1,\xi_2,\dots,\xi_r]^\Delta = \varprojlim 
     \Z/p\Z[\xi_1,\xi_2,\dots,\xi_r]_{<p}^{C_p}\ .$$
     See \cite{shank conj} for a discussion of, and examples of this approach.

    Here we use a simple combinatorial  method.  This is a generalization of the method used in \cite{CSW}
    where we considered a question related to Weitzenb\"ock derivations with Jordan blocks of size 2.

  \section{Main Theorem}
  \begin{theorem}[Main Theorem]\label{main theorem}
    Suppose the Weitzenb\"ock derivation acts on on a polynomial ring
    $R$ via a Jordan matrix on $R_1$ consisting entirely of Jordan blocks of size 3.
    Write $$R = K[x_1,y_1,z_1,x_2,y_2,z_2,\dots,x_n,y_n,z_n]$$ and
      write
      $$\Delta = \sum_{i=1}^{n} \left( x_{i}\frac{\partial}{\partial y_i} + y_{i}\frac{\partial}{\partial z_i}\right)\ .$$
    Then the algebra of constants $R^\Delta$
     is minimally generated as a  $K$-algebra by the following elements:
    \begin{enumerate}
      \item $f_{(i)} := x_i$ where $1 \leq i \leq n$;
      \item $f_{(i,j)} := x_i y_j - x_j y_i$ where $1 \leq i < j \leq n$;
      \item $g_{(i,j)} := x_i z_j - y_i y_j + z_i x_j$ where $1 \leq i \leq j \leq n$;
      \item $g_{(i,j,k)} = \det \left(
          \begin{matrix}
              x_i & y_i & z_i\\
              x_j & y_j & z_j\\
              x_k & y_k & z_k
          \end{matrix}       \right)$ where $1 \leq i < j < k \leq n$.
   \end{enumerate}
  \end{theorem}


  \begin{remark}
    Note that if the Weitzenb\"ock derivation $\Delta$ acts on a polynomial ring $P$ via a Jordan matrix on $P_1$
    consisting of blocks of size at most 3 then we have a surjective algebra homomorphism
    $\psi:R \to P$ which commutes with the action of $\Delta$.   Then $\psi:R^{\Delta} \to P^\Delta$ is also surjective
    and $\Psi(\G)$ forms a generating set for $P^\Delta$.
  \end{remark}


       We will work with monomial orders.  For a discussion of lead monomials and monomial
  orders we refer the reader to \cite[Ch.~2]{CLO}.   All tensor products are over the base field $K$.

     Let $\G$ denote the set of elements (1)--(4) listed in Theorem~\ref{main theorem}.
      It is easy to verify that each of these elements is annihilated by $\Delta$.
      Furthermore, by considering degrees it is easy to show that the
      elements of $\G$ minimally generate some $K$-algebra $Q$.
      We begin by sketching the main steps in our proof that $Q=R^\Delta$.

      Suppose that $h\in R^\Delta$ is a homogeneous polynomial of degree $d$.
      We consider the polynomial ring
      $$S = K[X_1,Y_1,Z_1,X_2,Y_2,Z_2,\dots,X_d,Y_d,Z_d].$$
      By abuse of notation we consider $\Delta$ to be a Weitzenb\"ock derivation on $S$
      via $\Delta(Z_i)=Y_i$, $\Delta(Y_i)=X_i$ and $\Delta(X_i)=0$ for $i=1,2,\dots,d$.

             We will use the classical techniques of polarization and restitution.
      In the next section we briefly describe these two techniques.
        For a detailed discussion in a general setting, we refer the reader to the excellent book of
        Procesi~(\cite[Ch.~3 \S 2]{P}).  
        We denote by $\Sml$ the $K$-subspace of $S$ spanned by the monomials of the form
        $\eta_1\eta_2\cdots \eta_d$ where $\eta_j \in \{X_j,Y_j,Z_j\}$ for all $1 \leq j \leq d$.
        Polarization is a $K$-linear operator
        $\Pol:R_d \to \Sml$ and
        restitution gives a $K$-linear operator $\Res: \Sml \to R_d$.
        Since both polarization and restitution commute with $\Delta$ we see that
        $\Pol:R_d^\Delta \to \Sml^\Delta = \Sml \cap S^\Delta$ and
        $\Res: \Sml^\Delta \to R_d^\Delta$.

      The full polarization of $h \in R_d^\Delta$ is $\Pol(h) \in \Sml^\Delta$.
      We find an explicit basis $\B{d}$ for $\Sml^\Delta$ as a $K$-vector space and so may write
      $H = \sum_{E \in \B{d}} c_E E$ for scalars $c_E \in K$.
      Then restituting $H$ yields $h = \sum_{E \in \B{d}} c_E \Res(E)$.
      The theorem then follows from the fact that each $\Res(E)$ may be expressed as a polynomial
       in the elements of $\G$.
      Since we may give these polynomial expressions explicitly and since we have an algorithm to
      compute the scalars $c_E$ we get an algorithm for expressing
      $h$ as a polynomial in the elements of $\G$.

      The main difficultly in the proof as outlined above is to find the basis
      $\B{d}$ of $\Sml^\Delta$.  We will construct a directed graph, in fact a rooted tree,
       $\Gamma$ and consider the set of paths $\Path{d}$ of length $d$ starting from the root.
       Naturally associated to each such path $\gamma$ we have a 
       monomial $\Lambda(\gamma) \in S$.
       We will construct a set map $\theta:\Path{d} \to \Sml^\Delta$ such that
        $\LM(\theta(\gamma))=\Lambda(\gamma)$ where $\LM$ denotes the leading monomial.
        Then showing that
       $\dim \Sml^\Delta = \# \Path{d}$ proves that $\theta(\Path{d})$ is a basis of
       $\Sml^\Delta$.


    \section{Polarization and Restitution}
      Here we give a brief description of the classical techniques of polarization and restitution.
        For a complete discussion in a general setting, see Procesi \cite[Ch.~3 \S2]{P}.
        Let $R = K[\eta_1,\eta_2,\dots,\eta_n]$ be a polynomial ring on $n$ variables with the standard $\N$-grading $R=\oplus_{d=1}^{\infty} R_d$.
        Given $d \geq 1$, let $M=(\eta_{i,j})_{1\leq i \leq n, 1\leq j \leq d}$ be a matrix of indeterminants and let
        $v = (v_1,v_2,\dots,v_d)^T$ be a column of indeterminants.  Given $f \in R_d$, we view 
        $f(Mv)$ as a polynomial with coefficients in the ring $S=K[\eta_{i,j} \mid 1\leq i \leq n, 1\leq j \leq d]$. 
        Note that the ring $S$  varies with the value of $d$.    
        We say that a monomial in $S$ is {\it multi-linear} if it not divisible by any quadratic monomial of the form $\eta_{i_1,j} \eta_{i_2,j}$.  
        A polynomial $H\in S$ {\it muti-linear} if it is a linear combination of muti-linear monomials.
        
        The {\it (full) polarization}
        $\Pol(f) \in S$ of $f$ is the coefficient of the monomial $v_1v_2\cdots v_d$ in $f(Mv)$.  
        In fact, $\Pol(f) \in \Sml$, the space of degree $d$ multi-linear polynomials in $S$, 
        i.e.,  every monomial of $\Pol(f)$ is divisible by exactly one of the indeterminants
        from each set $\{\eta_{i,j} \mid 1 \leq i \leq n\}$ for every $j$ with $1 \leq j \leq d$.
        Hence $\Sml$ is the $K$-vector space span of $\{\eta_{i_1,1} \eta_{i_2,2} \cdots \eta_{i_d,d}\}$. 
        We also consider the {\it restitution} map.  This is the algebra homomorphism
        $\Res : S \to R$ determined by $\Res(\eta_{i,j}) = \eta_i/d!$.  
         Thus we have a full polarization operator and a restitution homomorphism for each component $R_d$ of $R$.

        For our purposes, we take $n=3m$ and 
        $$R=K[x_1,y_1,z_1,x_2,y_2,z_2,\dots,x_m,y_m,z_m]$$ and
        $$S = K[X_1,Y_1,Z_1,X_2,Y_2,Z_2,\dots,Z_{md}].$$
        Writing $i=3k+r$ with $r \in{1,2,3}$ we relabel the $\eta_i$ and the $\eta_{i,j}$ via\\
         $$\eta_i = \begin{cases}  
                               x_{k+1}, &\text{if }r=1;\\
                               y_{k+1}, &\text{if }r=2;\\
                               z_{k+1}, &\text{if }r=3;
                        \end{cases} \text{ and }
         \eta_{i,j} = \begin{cases}  
                               X_{kd+j}, &\text{if }r=1;\\
                               Y_{kd+j}, &\text{if }r=2;\\
                               Z_{kd+j}, &\text{if }r=3.
                          \end{cases}$$
                        
             By abuse of notation we consider $\Delta$ to be a Weitzenb\"ock derivation on $S$
      by declaring that $\Delta(Z_i)=Y_i$, $\Delta(Y_i)=X_i$ and $\Delta(X_i)=0$ for $i=1,2,\dots,d$.


        In this notation the restitution map $\Res: S \to R$ is determined by $\Res(X_k)=x_\ell/d!$, $\Res(Y_k)=y_\ell/d!$ and $\Res(Z_k)=z_\ell/d!$
        where $\ell=\lceil k/d \rceil$.  
        With this notation a monomial in $S$ is multi-linear if the indeterminants dividing it have distinct subscripts.
     
          The following theorem summarizes the properties of polarization and restitution we will use.
             \begin{theorem}\
           \begin{enumerate}
             \item $\Pol : R_d \to \Sml$ is a $K$-linear operator.       \label{one}
              \item $\Res : S \to R$ is an algebra homomorphism.     \label{two}
              \item $\Res(\Pol(f))=f$ for all $f \in R_d$.                       \label{three}
              \item $\Pol$ and $\Res$ commute with $\sigma$.         \label{four}
              \item $\Pol$ and $\Res$ commute with $\Delta$.           \label{five}
              \item $\Pol : R_d^\Delta \to \Sml^\Delta$.                       \label{six}
              \item $\Res : \Sml^\Delta \to R_d^\Delta$.                     \label{seven}
             \end{enumerate}
        \end{theorem}

     \begin{proof}
        Proofs of statements (\ref{one})-(\ref{four}) may be found in \cite[Ch.3 \S2]{P}.
        Statement (\ref{five}) may be checked directly (using induction on degree) or we may use the fact 
        that polarization and restitution commute with $\sigma$.  Combining this with
        $\Delta = \ln(\sigma) = \ln (1+ (\sigma-1)) = \sum_{n=1}^\infty \frac{(-1)^{n+1}}{n} (\sigma-1)^n$ yields (\ref{five}).
        Note that this power series expansion for $\Delta$ is finite since $\sigma-1$ is nilpotent.     
        Statements (\ref{six}) and (\ref{seven}) follow immediately from statements (\ref{one}), (\ref{two}) and (\ref{five}).
     \end{proof}

  \section{Tensor Products of Jordan Matrices}
We seek to find the Jordan form for $\Delta$ on $\otimes^d V_3$ where $d$ is a positive integer.
Let $J_n(\lambda)$ denote the $n \times n$ Jordan matrix with a single Jordan block
and eigenvalue $\lambda$.

  The derivation $\Delta$ on $\Sml \cong \otimes^dV_3$ has Jordan decomposition
  $$\oplus_{k=1}^\infty \mu^d(k) J_k(0)$$ for some integers $\mu^d(k) \in \N\ .$
  Here we write $t J_k(0)$ to denote the direct sum of $t$ Jordan blocks of size $k$ and eigenvalue 0.

 It is not hard to see that the action of $\sigma=e^\Delta$ on $\otimes^d V_3$ has Jordan matrix given
 by $\oplus_{k=1}^\infty \mu^d(k) J_k(1)$.  In particular, the matrix of $\sigma$ on $V_3$ has Jordan form $J_3(1)$.
 To determine the numbers $\mu^d(k)$ we work with $\sigma$ rather than with $\Delta$ directly.
%
%
%
Hence we need to find the Jordan form for the Kronecker power $\otimes^d J_3(1)$.  To do this inductively
it suffices to decompose the Kronecker  product
$J_m(1) \otimes J_n(1)$ into a sum of Jordan blocks.

 The question of the Jordan decomposition of the Kronecker product of Jordan matrices
 was considered early in the last century.
The following theorem which provides the solution for the Jordan decomposition of $J_n(\lambda) \otimes J_m(\mu)$was enunciated
at that time (see \cite{A,L,R}).  However, it was not until rather later that a correct proof of this result \cite{Br,MR} was given.
For a discussion of the history of this problem see \cite{Br} or \cite{shank conj}.

     \begin{theorem}\label{Jordan product decomp}   Let $1 \leq m \leq n$.  Then
     $$J_m(1) \otimes J_n(1) = J_{n-m+1}(1) \oplus J_{n-m+3}(1) \oplus J_{n-m+5}(1) \oplus \dots \oplus J_{n+m-1}(1)\ .$$
  \end{theorem}

This yields the following.

\begin{lemma}\label{recursive mu}
 Suppose $k$ is an odd positive integer.  Then
  $$
      \mu^0(k) = 
                         \begin{cases}
                               1,&\text{if }k=1;\\
                               0,&\text{if } k\neq 1,
                          \end{cases}
      $$
and
  $$
  \mu^{d+1}(k) =
       \begin{cases}
             \mu^d(3),                                              & \text{if } k=1;\\
             \mu^d(k-2) + \mu^d(k) + \mu^d(k+2),  & \text{if } 3 \leq k.\\
        \end{cases}
  $$
for $d \geq 1$.
\end{lemma}

\section{The Representation Graph}

    In this section we introduce a directed graph $\Gamma$ which encodes
    the Jordan decomposition of tensor powers of $J_3(1)$.
    In order to simplify the exposition, we will consider $\Gamma$ as
    embedded in the $xy$-plane in the first quadrant.

    $\Gamma$ is tree with root at the point $(1,0)$.  The vertices of $\Gamma$
    are the integer lattice points $(k,d)$ in the first quadrant with $k$ odd and which lie
    on or above the line $y=x-1$, i.e., the points $(2a+1,d)$ with $a,d \in \Z$ and $0 \leq 2a \leq d$.
    We will also attach labels to the edges of $\Gamma$.
   Every vertex $(2a+1,d)$ of $\Gamma$ has a directed edge going up and to the right to
   the vertex at $(2a+3,d+1)$.  We label this edge with the symbol $X_{d+1}$.
   If $2a+1 \ge 3$ there is also a directed edge going straight up from $(2a+1,d)$ to $(2a+1,d+1)$.
   This vertical edge is labelled $Y_{d+1}$.  Finally , if $2a+1 \geq 3$ there is also an edge
   from $(2a+1,d)$ up and leftward to $(2a-1,d+1)$.  This edge is labelled $Z_{d+1}$.
   Note that the edge labels are not distinct.

   Consider a path in the directed graph from the root $(1,0)$ to a vertex $(2a+1,d)$.
   Reading the edge labels of this path yields $d$ labels, each from the set
   $\{X_1,Y_1,Z_1,X_2,Y_2,Z_2,\dots,X_d,Y_d,Z_d\}$.  Furthermore each of the subscripts
   $1,2,\dots,d$ occurs exactly once.   Multiplying these labels together yields a
    monomial of degree $d$ in $\Sml\cong\otimes^d V_3$.
      Thus to each path $\gamma$ of length $d$ originating from the root we have
    associated a monomial which we denote by $\Lambda(\gamma)$.
      We call the monomials
    which can be constructed in this manner, {\em path monomials} and we denote by $M_d$ the
    path monomials arising from paths of length $d$.
We will show that these path monomials are exactly the lead monomials of elements of
$\Sml^\Delta \cong (\otimes^d V_3)^\Delta$.

  We begin by counting paths in $\Gamma$.
 Let $\nu^d(k)$ denote the number of distinct paths in $\Gamma$ from the root $(1,0)$
 to the vertex $(k,d)$.
   With this notation we have the following lemma whose proof is left to the reader.

  \begin{lemma}  \label{recursive nu}
    $$\nu^0(k) = 
                          \begin{cases}
                               1,&\text{if }k=1;\\
                               0,&\text{if } k\neq 1,
                          \end{cases}
    $$
    and
    $$\nu^{d+1}(k) =
        \begin{cases}
               \nu^d(3),                                             & \text{if } k=1;\\
               \nu^d(k-2) +  \nu^d(k) + \nu^d(k+2),  & \text{if } 3 \leq k.\\
        \end{cases}$$  
  \end{lemma}

The following corollary is immediate.
  \begin{corollary}\label{counting corollary}
  For all $d \in \N$  and all odd positive integers $k$ we have
    $$\mu^d(k) = \nu^d(k)\ .$$
  \end{corollary}

\section{A Vector Space Basis for $\Sml^\Delta$}

For the remainder of this paper, $d$ is a fixed positive integer.

We define the following multi-linear elements of $S^\Delta$:
  \begin{enumerate}
      \item $F_{\{i\}} := X_i$ where $1 \leq i \leq md$;
      \item $F_{\{i,j\}} := X_i Y_j - X_j Y_i$ where $1 \leq i < j \leq md$;
      \item $G_{\{i,j\}} := X_i Z_j - Y_i Y_j + Z_i Y_j$ where $1 \leq i < j \leq md$;
      \item $G_{\{i,j,k\}} = \det \left(
          \begin{matrix}
              X_i & Y_i & Z_i\\
              X_j & Y_j & Z_j\\
              X_k & Y_k & Z_k
          \end{matrix}       \right)$ where $1 \leq i < j < k \leq md$.
   \end{enumerate}

 From these elements we inductively construct two families of multi-linear elements of $S^\Delta$ as follows.
      \begin{enumerate}
         \item
        $F_{\{i_1,i_2,\dots,i_t\}} := F_{\{i_2,i_4,i_5,i_6\dots,i_t\}}G_{\{i_1,i_3\}} - F_{\{i_1,i_4,i_5,i_6\dots,i_t\}}G_{\{i_2,i_3\}}$\\
         where $1 \leq i_1 < i_2 < \dots < i_t \leq md$ and $t \geq 3$.
 \item
        $G_{\{i_1,i_2,\dots,i_t\}} := G_{\{i_2,i_4,i_5,i_6\dots,i_t\}}G_{\{i_1,i_3\}} - G_{\{i_1,i_4,i_5,i_6\dots,i_t\}}G_{\{i_2,i_3\}}$\\
            where $1 \leq i_1 < i_2 < \dots < i_t \leq md$ and $t \geq 4$.
      \end{enumerate}

      We denote the union of these families by $\B{}$ and we write $\B{d}$ to denote those products of elements of $\B{}$ which have
      total degree $d$ and lie in  $\Sml$.

     We use the lexicographic order on $S$ determined by
$$Z_{md} > Y_{md} > X_{md} > Z_{md-1} > Y_{md-1} > X_{md-1} > \dots > Z_1 > Y_1 > X_1\ .$$
The following lemma exhibits the two largest terms for elements of $\B{}$.
\begin{lemma}
    Let $1 < i_1 < i_2 \dots < i_t \leq md$.
    Then
  \begin{enumerate}
     \item $F_{\{i_1,i_2,\dots,i_t\}} = X_{i_1} Y_{i_2} Y_{i_3} Y_{i_4}  \cdots Y_{i_t}
                                                      - Y_{i_1} X_{i_2} Y_{i_3} Y_{i_4} \cdots Y_{i_t} + \lot,$\\ if $t \geq 2$.
     \item \label{g_part}
     $G_{\{i_1,i_2,\dots,i_t\}} = X_{i_1} Y_{i_2} Y_{i_3} Y_{i_4}  \cdots Y_{i_{t-1}} Z_{i_t}
                                                      - Y_{i_1} X_{i_2} Y_{i_3} Y_{i_4} \cdots Y_{i_{t-1}} Z_{i_t}\\ + \lot,$ if $t \geq 3$.
  \end{enumerate}
  where $\lot$\ denotes lower order terms.
\end{lemma}

\begin{proof}
  The proof is by induction on $t$.  The result is straightforward to verify for $t=2,3$.  For higher values of $t$ we have
  (using induction)
  \begin{align*}
     F&_{\{i_1,i_2,\dots,i_t\}} = F_{\{i_2,i_4,i_5\dots,i_t\}}G_{\{1,3\}} - F_{\{i_1,i_4,i_5,i_6\dots,i_t\}}G_{\{i_2,i_3\}}\\
            &= (X_{i_2}Y_{i_4} Y_{i_5} \cdots Y_{i_t} - Y_{i_1} X_{i_2} Y_{i_3} \cdots Y_{i_t} + \lot)
                      (X_{i_1}Z_{i_3} - Y_{i_1}Y_{i_3} + Z_{i_1}Y_{i_3})\\
                      &\quad- (X_{i_1}Y_{i_4} Y_{i_5} \cdots Y_{i_t} - Y_{i_1} X_{i_4} Y_{i_5} \cdots Y_{i_t} + \lot)
                      (X_{i_2}Z_{i_3} - Y_{i_2}Y_{i_3} + Z_{i_2}Y_{i_3})\\
        &= X_{i_1}X_{i_2}Z_{i_3}Y_{i_4}Y_{i_5} \cdots Y_{i_t} - X_{i_1}Y_{i_2}Z_{i_3}Y_{i_4}Y_{i_5} \cdots Y_{i_t} + \lot\\
        &\quad -Y_{i_1}X_{i_2}Y_{i_3}Y_{i_4}Y_{i_5} \cdots Y_{i_t} + Y_{i_1}Y_{i_2}Y_{i_3}X_{i_4}Y_{i_5} \cdots Y_{i_t}+\lot\\
        &\quad + Z_{i_1}X_{i_2}X_{i_3}Y_{i_4}Y_{i_5} \cdots Y_{i_t} - Z_{i_1}Y_{i_2}X_{i_3}X_{i_4}Y_{i_5} \cdots Y_{i_t}+ \lot\\
        &\quad - X_{i_1}X_{i_2}Z_{i_3}Y_{i_4}Y_{i_5} \cdots Y_{i_t} + Y_{i_1}X_{i_2}Z_{i_3}X_{i_4}Y_{i_5} \cdots Y_{i_t}+\lot\\
        &\quad + X_{i_1}X_{i_2}Z_{i_3}Y_{i_4}Y_{i_5} \cdots Y_{i_t} - Y_{i_1}Y_{i_2}Y_{i_3}X_{i_4}Y_{i_5} \cdots Y_{i_t} +\lot\\
        &\quad  - X_{i_1}Z_{i_2}X_{i_3}Y_{i_4}Y_{i_5} \cdots Y_{i_t} + Y_{i_1}Z_{i_2}X_{i_3}X_{i_4}Y_{i_5} \cdots Y_{i_t} + \lot\\
        &= X_{i_1}Y_{i_2}Y_{i_3}Y_{i_4}Y_{i_5} \cdots Y_{i_t} - Y_{i_1}X_{i_2}Y_{i_3}Y_{i_4}Y_{i_5} \cdots Y_{i_t} + \lot
     \end{align*}
         The proof for (\ref{g_part}) is similar with the cases $t=3,4$ being easily verified.
  \begin{align*}
     G&_{\{i_1,i_2,\dots,i_t\}} = G_{\{i_2,i_4,i_5\dots,i_t\}}G_{\{{i_1},{i_3}\}} - G_{\{i_1,i_4,i_5,i_6\dots,i_t\}}G_{\{i_2,i_3\}}\\
            &= (X_{i_2}Y_{i_4} Y_{i_5} \cdots Y_{i_t} - Y_{i_1} X_{i_2} Y_{i_3} \cdots Y_{i_{t-1}}Z_{i_t} + \lot)
                      (X_{i_1}Z_{i_3} - Y_{i_1}Y_{i_3} + Z_{i_1}Y_{i_3})\\
                      &\quad- (X_{i_1}Y_{i_4} Y_{i_5} \cdots Y_{i_t} - Y_{i_1} X_{i_4} Y_{i_5} \cdots Y_{i_{t-1}}Z_{i_t} + \lot)
                      (X_{i_2}Z_{i_3} - Y_{i_2}Y_{i_3} + Z_{i_2}Y_{i_3})\\
        &= X_{i_1}X_{i_2}Z_{i_3}Y_{i_4}Y_{i_5} \cdots Y_{i_{t-1}}Z_{i_t} - X_{i_1}Y_{i_2}Z_{i_3}Y_{i_4}Y_{i_5} \cdots Y_{i_{t-1}}Z_{i_t} + \lot\\
        &\quad -Y_{i_1}X_{i_2}Y_{i_3}Y_{i_4}Y_{i_5} \cdots Y_{i_{t-1}}Z_{i_t} + Y_{i_1}Y_{i_2}Y_{i_3}X_{i_4}Y_{i_5} \cdots Y_{i_{t-1}}Z_{i_t}+\lot\\
        &\quad + Z_{i_1}X_{i_2}X_{i_3}Y_{i_4}Y_{i_5} \cdots Y_{i_{t-1}}Z_{i_t} - Z_{i_1}Y_{i_2}X_{i_3}X_{i_4}Y_{i_5} \cdots Y_{i_{t-1}}Z_{i_t}+ \lot\\
        &\quad - X_{i_1}X_{i_2}Z_{i_3}Y_{i_4}Y_{i_5} \cdots Y_{i_{t-1}}Z_{i_t} + Y_{i_1}X_{i_2}Z_{i_3}X_{i_4}Y_{i_5} \cdots Y_{i_{t-1}}Z_{i_t}+\lot\\
        &\quad + X_{i_1}X_{i_2}Z_{i_3}Y_{i_4}Y_{i_5} \cdots Y_{i_{t-1}}Z_{i_t} - Y_{i_1}Y_{i_2}Y_{i_3}X_{i_4}Y_{i_5} \cdots Y_{i_{t-1}}Z_{i_t} +\lot\\
        &\quad  - X_{i_1}Z_{i_2}X_{i_3}Y_{i_4}Y_{i_5} \cdots Y_{i_{t-1}}Z_{i_t} + Y_{i_1}Z_{i_2}X_{i_3}X_{i_4}Y_{i_5} \cdots Y_{i_{t-1}}Z_{i_t} + \lot\\
        &= X_{i_1}Y_{i_2}Y_{i_3}Y_{i_4}Y_{i_5} \cdots Y_{i_{t-1}}Z_{i_t} - Y_{i_1}X_{i_2}Y_{i_3}Y_{i_4}Y_{i_5} \cdots Y_{i_{t-1}}Z_{i_t} + \lot
     \end{align*}
\end{proof}

  In Proposition~\ref{Path basis}, we will prove that $\B{d}$ is a vector space basis for $\Sml^\Delta$.

\section{Definition of $\theta$ and $\phi$}

Recall that $M_d$ denotes the set of path monomials arising from paths of length $d$ and
that $\Lambda : \Path{d} \to M_d$.
We will define set maps
$$\phi:M_d \to \Sml^{\Delta}$$ and
$$\theta=\phi\circ\Lambda:\Path{d} \to \Sml^{\Delta}$$
such that $\LM(\theta(\gamma)) = \Lambda(\gamma)$.
Furthermore, $\theta(\gamma)$ will be a product of elements from $\B{}$ and so $\theta(\gamma) \in \B{d}$.

 Let $\alpha$ be a monomial in $S$.
Put $\supp_X(\alpha) = \{i \mid X_i\text{ divides }\alpha\}$,
$\supp_Y(\alpha) = \{i \mid Y_i\text{ divides }\alpha\}$,
$\supp_Z(\alpha) = \{i \mid Z_i\text{ divides }\alpha\}$ and
$\supp(\alpha) = \supp_X(\alpha) \cup \supp_Y(\alpha) \cup \supp_Z(\alpha)$.
Suppose that $\alpha \in M_d$.  Note that $\supp(\alpha)$ is in the interval of integers $[1,d]$.


We define $\phi(\alpha)$ as follows.   Write $\supp_Z(\alpha) = \{k_1, k_2, \dots, k_s\}$
where $k_1 < k_2 < \dots < k_s$.  We begin by defining $$\{i_1,i_2,\dots,i_s\} \subset \supp_X(\alpha)$$
with $i_1 > i_2 > \dots > i_s$.  Let 
$$i_1 := \max\{i \in \supp_X(\alpha) \mid i < k_1\}$$
 and put
$I_1 := [i_1,k_1]$ (an interval in $\N$).
Let $$i_2 := \max \{i \in \supp_X(\alpha) \mid i < i_1\}$$ and $I_2 := [i_2,k_2] \setminus I_1$.
In general, $$i_q := \max \{i \in \supp_X(\alpha) \mid i < i_{q-1}\}$$ and
$$I_q := [i_q,k_q] \setminus (\sqcup_{\ell=1}^{q-1} I_\ell)$$ for $q=2,\dots,s$.

 Define $\ba := \alpha / (\prod_{\ell=1}^s \LM(G_{I_\ell}))$.
 Then $\supp(\ba) = [1,d] \setminus (\sqcup_{\ell=1}^s I_\ell)$.
 For each $j \in \supp_Y(\ba)$ we define 
 $$\succ(j) := \min\{i \in \supp(\ba) \mid i > j\}.$$
 Let 
 $$\{j_1,j_s,\dots,j_t\} = \{ j \in \supp_Y(\ba) \mid \succ(j) \notin \supp_Y(\ba)\}$$
  where
 $j_1 < j_2 < \dots < j_t$.   Next we define 
 $$\{i'_1,i'_2,\dots,i'_t\} \subset \supp_X(\ba)\}$$
  with
 $i'_1 > i'_2 > \dots > i'_t$ as follows.
  Let 
  $$i'_1 := \max\{ i \in \supp_X(\ba) \mid i < j_1\}$$
   and $I'_1 := [i'_1,j_1] \cap \supp(\ba)$.
  Let 
  $$i'_2 := \max\{ i \in \supp_X(\ba)  \mid i < i_2\}$$ and 
  $$I'_2 := ([i'_2,j_2] \setminus I'_1) \cap \supp(\ba).$$
  In general, $$i'_q := \max\{ i \in \supp_X(\ba) \mid i_q < i_{q-1}\}$$
  and $$I'_q := ([i'_q,j_q] \setminus (\sqcup_{\ell=1}^{q-1} I'_\ell)) \cap \supp(\ba).$$
  We put 
  $$I'' := [1,d] \setminus ((\sqcup_{\ell=1}^s I_\ell) \sqcup (\sqcup_{\ell=1}^t I'_\ell)).$$
  Note that $I'' \subseteq \supp_X(\ba)$.

  Finally, we define
  $$\phi(\alpha) := (\prod_{\ell=1}^s G_{I_\ell})\cdot (\prod_{\ell=1}^t F_{I'_\ell}) \cdot (\prod_{i \in I''} F_{\{i\}})\ .$$

  For each $\ell=1,2,\dots,s$ we have $\min I_\ell = i_\ell$ and $\max I_\ell = k_\ell$.
  Define 
  $$J_\ell := \{ j \in I_\ell \mid i_\ell < j < k_\ell\}.$$
    Then
  $\LM(G_{I_\ell}) = X_{i_\ell} \cdot (\prod_{j \in J_\ell} Y_j) \cdot Z_{k_\ell}$.

  For each $\ell=1,2,\dots,t$ we have $\min I'_\ell = i'_\ell$ and $\max I'_\ell = j_\ell$.
  Define $J'_\ell := I'_\ell \setminus \{i'_\ell\}$.
  Then $\LM(F_{I'_\ell}) = X_{i'_\ell} \cdot (\prod_{j \in J'_\ell} Y_j)$.

  Therefore $\LM(\phi(\alpha)) = \alpha$ as required.
  Furthermore, $\phi(\alpha)$ is a product of elements of $\B{}$ and thus $\phi(\alpha)\in \B{d}$.

\section{Proof of the Main Theorem}

We now prove that degree $d$ path monomials are exactly the lead monomials of degree $d$ multi-linear $\Delta$-constants,
Thus
$\theta$ provides a bijection between paths of length $d$ and a basis of the degree $d$ multi-linear $\Delta$-constants,
$\Sml^\Delta$

\begin{proposition}\label{Path basis}
   $$
    M_d = \{ \LM(f) \mid  \deg(f)=d, f \in \Sml^{\Delta}\}\ .
   $$
Moreover, $\{\theta(\gamma) \mid \gamma \in \Path{d}\}$ is a vector space basis for
$\Sml^\Delta$.
\end{proposition}
\begin{proof}
   Since $\LM(\phi(\alpha)) = \alpha$ for all $\alpha \in M_d$ it follows that
   $M_d \subseteq \{ \LM(f) \mid  \deg(f)=d, f \in \Sml^{\Delta}\}$.
  Since 
     \begin{align*}
          \# M_d &= \#\Path{d}~
                       = \sum_{k {\rm\ odd}, k \leq d} \nu_k(d)\\
            & =  \sum_{k {\rm\ odd}, k \leq d} \mu_k(d)
             = \dim \Sml^{\Delta}\\
            & =  \# \{\LM(f) \mid \deg(f)=d, f \in \Sml^{\Delta}\},
      \end{align*}
      we see that
        $M_d = \{ \LM(f) \mid  \deg(f)=d, f \in \Sml^{\Delta}\}$.

    Furthermore, $\LM(\phi(\alpha)) = \alpha$ for all $\alpha \in M_d$ implies that
    the set $\phi(M_d)=\theta(\Path{d})$ is linearly independent.  Therefore $\theta(\Path{d})$ is a basis of $\Sml^{\Delta}$.
\end{proof}

\begin{remark}
  In fact it is possible to show that if $\gamma$ is a path from the root to $(k,d)$ then $\theta(\gamma)$ is
  an eigenvector corresponding to a Jordan block of size $k$.
\end{remark}

  Suppose $h$ lies in the algebra of constants $R^\Delta$.   Further suppose that $h$ is homogeneous
of degree $d$. 
Let $H$ denote the full polarization $\Pol(h)$ of $h$.  Then $H \in \Sml^{\Delta}$.
Thus $H = \sum_{E \in \B{d}} c_E E$ for constants $c_E \in K$.   In fact we  may compute these coefficients $c_E$
as follows.  We know $\LM(H) = \Lambda(\gamma_1)=\LM(\theta(\gamma_1))$ for some $\gamma_1 \in \Path{d}$.  Then the lead term of $H$
is $c_{\gamma_1} \LM(\theta(\gamma_1))$ for some scalar $c_{\gamma_1}$.
Put $H_2 = H- c_{\gamma_1}\theta({\gamma_1})$.    Then $H_2 \in  \Sml^{\Delta}$ and so $\LM(H_2) = \LM(\theta(\gamma_2))$ for some $\gamma_2 \in \Path{d}$.
Hence the lead term of $H_2$ is $c_{\gamma_2} \LM(\theta(\gamma_2))$ for some scalar $c_{\gamma_2}$.
Put $H_3 = H_2- c_{\gamma_2}\theta(\gamma_2)$.  Continuing in this manner we construct a sequence of polynomials
$H=H_1,H_2,H_3,\dots$ with $\LM(H_1) > \LM(H_2) > \LM(H_3) > \dots$.  Since such a decreasing sequence of monomials must be finite,
we will eventually find $H_{r+1}=0$ for some $r$.  Thus $H = \sum_{i=1}^r c_{\gamma_i} \theta(\gamma_i)$. 

Then $h =  \Res(H)= \sum_{i=1}^r c_{\gamma_i} \Res(\theta(\gamma_i))$ where each $\theta(\gamma_i) \in \B{d}$.
Each $E \in \B{d}$ is of the form $E=\prod_{I \in A} F_I \cdot \prod_{I' \in A'} G_{I'}$ for some index sets
$A$ and $A'$ contained in the power set of $\{1,2,\dots,md\}$.  Thus $\Res(E) = \prod_{I \in A} \Res(F_I) \cdot \prod_{I' \in A'} \Res(G_{I'})$.

\begin{lemma}  Let $Q$ denote the $K$-algebra generated by $\G$.
  \begin{enumerate}
    \item $\Res(F_I) \in Q$ for all $I \subseteq \{1,2,\dots,md\}$.
    \item $\Res(G_{I'}) \in Q$ for all $I' \subseteq \{1,2,\dots,md\}$ with $\# I' \geq 2$.
  \end{enumerate}
\end{lemma}

\begin{proof}
  We prove the second assertion first.   We proceed by induction on the cardinality of the set $I'$.
  If $I' =\{i,j\}$ with $i<j$ then $\Res(G_{I'})=x_k z_\ell ?y_k y_\ell + z_k y_\ell$ where $k = \lceil i/d \rceil$ and
   $\ell = \lceil j/d \rceil$.  Thus  $\Res(G_{I'}) \in \G$.

If $I' = \{i,j,k\}$ with $i<j<k$ then $\Res(G_I) = g_{(a,b,c)}$ where $a = \lceil i/d \rceil$, $b = \lceil j/d \rceil$
  and $c = \lceil k/d \rceil$.    Thus either  $\Res(G_I)\in \G$ or  $\Res(G_I)=0$. 

 Assume, by induction, that the that the second assertion holds if $\#I' < t$ (where $t \geq 4$)
 and that $\#I'=t$.     Write $I=\{i_1,i_2,\dots,i_t\}$ where $i_1 < i_2 < \dots < i_t$.  
  Then 
  $$G_{I'} = G_{I_1} G_{I_2} - G_{I_3} G_{I_4}$$ where 
  \begin{align*} 
         I_1 &= \{i_2,i_4,i_5,i_6,\dots,i_t\}, \qquad I_2 = \{i_1,i_3\},\\
         I_3 &= \{i_1,i_4,i_5,i_6,\dots,i_t\} \text{ and }I_4 = \{i_2,i_3\}.
     \end{align*}
  Therefore $$\Res(G_I) = \Res(G_{I_1})\Res(G_{I_2}) - \Res(G_{I_3})\Res(G_{I_4}).$$
  Since $\Res(G_{I_1}),\Res(G_{I_2}),\Res(G_{I_2}),\Res(G_{I_4}) \in Q$ by the induction hypothesis, this implies that
  $\Res(G_{I'}) \in Q$.

  We prove the first assertion by induction on $\# I$.
  If $I =\{i\}$ then $\Res(F_I) = x_i \in \G$.
  
  If $I =\{i,j\}$ with $i <j$ then $\Res(F_I) = x_k y_\ell - x_k y_\ell$ where $k = \lceil i/d \rceil$ and
   $\ell = \lceil j/d \rceil$.
    Therefore either $\Res(F_I) \in \G$ or $\Res(F_I)=0$.

  Assume, by induction, that the first assertion holds if $\# I < t$ (where $t\geq 3$) and that $\#I=t$.
  Write $I=\{i_1,i_2,\dots,i_t\}$ where $i_1 < i_2 < \dots < i_t$.  
  Then 
  $$F_I = F_{I_1} G_{I_2} - F_{I_3} G_{I_4}$$ where 
  \begin{align*} 
         I_1 &= \{i_2,i_4,i_5,i_6,\dots,i_t\}, \qquad I_2 = \{i_1,i_3\},\\
         I_3 &= \{i_1,i_4,i_5,i_6,\dots,i_t\} \text{ and }I_4 = \{i_2,i_3\}.
     \end{align*}
  Therefore $$\Res(F_I) = \Res(F_{I_1})\Res(G_{I_2}) - \Res(F_{I_3})\Res(G_{I_4}).$$
  But $\Res(G_{I_2}),\Res(G_{I_4}) \in Q$ by the second assertion and
  $\Res(F_{I_1}),\Res(F_{I_3}) \in Q$ by the induction hypothesis.
  Therefore $\Res(F_I) \in Q$.
 \end{proof}

This lemma completes the proof that $R^\Delta$ is the $K$-algebra generated by $\G$.
Moreover, the proof of the above lemma provides an inductive algorithm for writing any element of $\B{}$
as a polynomial in the elements of $\G$.

\section{Higher degrees of Nilpotency}
In principal the method used here should work for Weitzenb\"ock derivations with Jordan blocks of size $k$ for any 
    fixed $k$.  However, for $k \geq 4$ the set of paths in the corresponding representation graph $\Gamma$ becomes 
    unmanageable.  In particular, here we have built all such paths up piecewise (in \S7) from two simple types of paths,
    corresponding to the two types of lead terms $X_{i_1}Y_{i_2}Y_{i_3}\cdots Y_{i_t}$ and
    $X_{i_1}Y_{i_2}Y_{i_3}\cdots Y_{i_{t-1}}Z_{i_t}$.  This has allowed us to generate the $\Delta$-constants 
    as an algebra using only two corresponding types ($F$ and $G$) of elements in $\Sml^\Delta$.  The large variability in the form of the paths 
    for the case $k=4$ seems to require an equally large collection $\Delta$-constants with corresponding lead terms.
    It seems that enumerating and constructing these $\Delta$-constants will be very difficult.

\section{Acknowledgements}
  I thank the anonymous referee for a number of very good suggestions which have improved the exposition and shortened some of the proofs.
 I thank Megan Wehlau for a number of useful discussions which led to this work.  The computer algebra program Magma \cite{magma}
   was very helpful in my early explorations of this problem.  

\end{document}